\documentclass [leqno] {amsart}
  
\usepackage{amssymb}
\usepackage{amsthm}
\usepackage{amscd}

\usepackage[all]{xy}
\usepackage{hyperref} 

\begin{document}

\theoremstyle{plain}
\newtheorem{proposition}[subsubsection]{Proposition}
\newtheorem{lemma}[subsubsection]{Lemma}
\newtheorem{corollary}[subsubsection]{Corollary}
\newtheorem{thm}[subsubsection]{Theorem}
\newtheorem{introthm}{Theorem}
\newtheorem*{thm*}{Theorem}
\newtheorem{conjecture}[subsubsection]{Conjecture}
\newtheorem{question}[subsubsection]{Question}
\newtheorem{fails}[subsubsection]{Fails}

\theoremstyle{definition}
\newtheorem{definition}[subsubsection]{Definition}
\newtheorem{notation}[subsubsection]{Notation}
\newtheorem{condition}[subsubsection]{Condition}
\newtheorem{example}[subsubsection]{Example}
\newtheorem{claim}[subsubsection]{Claim}

\theoremstyle{remark}
\newtheorem{remark}[subsubsection]{Remark}

\numberwithin{equation}{subsection}

\newcommand{\eq}[2]{\begin{equation}\label{#1}#2 \end{equation}}
\newcommand{\ml}[2]{\begin{multline}\label{#1}#2 \end{multline}}
\newcommand{\mlnl}[1]{\begin{multline*}#1 \end{multline*}}
\newcommand{\ga}[2]{\begin{gather}\label{#1}#2 \end{gather}}
\newcommand{\mat}[1]{\left(\begin{smallmatrix}#1\end{smallmatrix}\right)}

\newcommand{\arir}{\ar@{^{(}->}}
\newcommand{\aril}{\ar@{_{(}->}}
\newcommand{\are}{\ar@{>>}}

\newcommand{\xr}[1] {\xrightarrow{#1}}
\newcommand{\xl}[1] {\xleftarrow{#1}}
\newcommand{\lra}{\longrightarrow}
\newcommand{\inj}{\hookrightarrow}

\newcommand{\mf}[1]{\mathfrak{#1}}
\newcommand{\mc}[1]{\mathcal{#1}}

\newcommand{\CH}{{\rm CH}}
\newcommand{\Gr}{{\rm Gr}}
\newcommand{\codim}{{\rm codim}}
\newcommand{\cd}{{\rm cd}}
\newcommand{\Spec} {{\rm Spec}}
\newcommand{\supp} {{\rm supp}}
\newcommand{\Hom} {{\rm Hom}}
\newcommand{\End} {{\rm End}}
\newcommand{\id}{{\rm id}}
\newcommand{\Aut}{{\rm Aut}}
\newcommand{\sHom}{{\rm \mathcal{H}om}}
\newcommand{\Tr}{{\rm Tr}}


\renewcommand{\P} {\mathbb{P}}
\newcommand{\Z} {\mathbb{Z}}
\newcommand{\Q} {\mathbb{Q}}
\newcommand{\C} {\mathbb{C}}
\newcommand{\F} {\mathbb{F}}

\newcommand{\OO}{\mathcal{O}}
\newcommand{\Di}{\mathfrak{Di}}
\newcommand{\divides}{\mid}
\newcommand{\notdivides}{\nmid} 
\newcommand{\coker}{{\rm coker}}
\newcommand{\Spf}{{\rm Spf}}

\title[Commutative formal groups arising from schemes]{Commutative formal groups arising from schemes}

\author{Andre Chatzistamatiou}
\address{Fachbereich Mathematik \\ Universit\"at Duisburg-Essen \\ 45117 Essen, Germany}
\email{a.chatzistamatiou@uni-due.de}

\thanks{This work has been supported by the SFB/TR 45 ``Periods, moduli spaces and arithmetic of algebraic varieties''}

\begin{abstract}
We prove the following criterion for the pro-representability of the deformation
cohomology of a commutative formal Lie group. Let $f$ be a flat morphism between noetherian schemes. Assume that the target of $f$ is flat over 
the integers. For a commutative formal Lie group $E$, we have the deformation 
cohomology of $f$ with coefficients in $E$ at our disposal. If the higher direct images
of the tangent space of $E$ are locally free and of finite rank then the deformation
cohomology is pro-representable by a commutative formal Lie group.    
\end{abstract}

\maketitle

\section*{Introduction}

Let $f:X\xr{} S$ be a flat and separated morphism. The formal completion 
$\widehat{{\rm Pic}}_{X/S}$ of the relative Picard sheaf can be described 
as the first deformation cohomology $\Phi^1(X/S)$ with coefficients in the 
formal completion $\hat{\mathbb{G}}_m$ of the multiplicative group. Artin 
and Mazur, in their paper \cite{Artin-Mazur}, study the higher degree analogues
$\Phi^r(X/S)$, motivated by the idea that these sheaves exhibit a strong tendency
to be pro-representable by formal Lie groups. Commutative formal Lie groups
provide interesting invariants over a field of positive characteristic but poor
invariants in characteristic zero where every commutative formal Lie group is isomorphic 
to a product of $\hat{\mathbb{G}}_m$.

The purpose of this paper is to give a convenient criterion for $\Phi^r(X/S)$ 
to be pro-representable when $S$ is flat over $\Spec(\Z)$. In order to state our theorem 
we need to introduce a bit of notation. We denote by
$\Phi^r(X/S,E)$ the deformation cohomology of a commutative formal Lie group $E$ over $X$ (see \ref{subsubsection-deformation-cohomology}) 
and $\Phi^r(X/S)=\Phi^r(X/S,\hat{\mathbb{G}}_m)$. The tangent space of $E$ is
denoted by $TE$, it is a locally free $\OO_X$-module; for example, $T\hat{\mathbb{G}}_m=\OO_X$. 

\begin{introthm}[Theorem \ref{thm-main}]\label{thm-intro-main}
  Let $f:X\xr{} S$ be a flat morphism between noetherian schemes. 
Suppose that $S$ is flat over $\Spec(\Z)$. Let $E$ be a commutative formal Lie group over $X$.
Suppose that $R^rf_*TE$ is locally free and of finite rank for all $r\geq r_0$. 
Then $\Phi^r(X/S,E)$ is pro-representable by a commutative formal Lie group over $S$ for all $r\geq r_0$.
\end{introthm}

In a first version of this paper, we proved Theorem~\ref{thm-intro-main} by using Cartier theory. It turns out 
that the Cartier module associated to $\Phi^r(X/S)$ is $H^r(X,\mathbb{W}\OO_X)$,
where $\mathbb{W}\OO_X$ is the sheaf of big Witt vectors, provided that $S$
is affine. The pro-representability of $\Phi^r(X/S)$ can be therefore reduced
to certain properties of $H^r(X,\mathbb{W}\OO_X)$ which are made precise by Cartier theory. 

In this version of the paper, we follow a quicker and more elegant approach via 
Lemma \ref{lemma-referee}, which was pointed out by the referee --
we are very grateful to the referee for sharing this proof with us.     

One motivation for the paper is the existence of a big de Rham-Witt 
complex $\mathbb{W}\Omega^{\bullet}_X$ for arbitrary schemes due to  
Hesselholt and Madsen \cite{HM}, \cite{H}. Here $\mathbb{W}\OO_X$ appears 
as a quotient of $\mathbb{W}\Omega^{\bullet}_X$, and in view of the Bloch-Illusie 
slope spectral sequence for smooth varieties over a perfect field one could 
expect that $H^*(X,\mathbb{W}\OO_X)$ captures an interesting part of $H^*(X,\mathbb{W}\Omega^{\bullet}_X)$.


In Section \ref{subsection-Gauss-Manin} we describe the relationship with the Gauss-Manin connection. If $S$ is a smooth scheme over $\Z$, and
$\mathfrak{X}$ is a commutative formal Lie group over $S$, then the Katz-Oda construction yields a connection 
$$
\nabla:H^i(\Omega^*_{\mathfrak{X}/S})\xr{} 
H^i(\Omega^*_{\mathfrak{X}/S})\otimes_{\OO_S}\Omega^1_{S/\Z}\qquad \text{for all $i$.}
$$
Since in the invariant 1-forms (see \ref{subsection-invariant-1-forms}) are automatically
closed, we may define $(H^1_{{\rm inv}}(\mathfrak{X}),\nabla)$ to be the smallest
subobject of $(H^1(\Omega^*_{\mathfrak{X}/S}),\nabla)$ that contains the invariant 
1-forms. 
As an application of a theorem of Stienstra \cite{Stienstra} we will prove the 
following statement. 

\begin{introthm}[Theorem \ref{thm-GM-simple}]
  Let $f:X\xr{} S$ be a smooth, projective morphism of relative dimension $r$.
Suppose that $S$ is smooth over $\Spec(\Z)$ and suppose that 
$R^jf_*\Omega^i_{X/S}$ is locally free for all $i,j$. 
Then $(H^1_{{\rm inv}}(\Phi^i(X/S)),\nabla)$ is a subquotient of the 
Gauss-Manin connection $(H^{2r-i}_{{\rm dR}}(X/S),\nabla)$ for all $i\geq 0$.
\end{introthm}

\subsubsection*{Acknowledgments}
Lemma \ref{lemma-referee} and its use for proving the main theorem was shown to us by the referee for  
International Mathematics Research Notices. We are very grateful to the referee for providing this argument.

\section{Deformation cohomology}
\subsection{}\label{subsubsection-formal-completion}
Let $X$ be a scheme. Let $E$ be a sheaf of abelian groups on the big Zariski site of $X$.
The \emph{formal completion} $\hat{E}$ of $E$ along its zero section is defined by
$$
\hat{E}(Z):=\ker(E(Z)\xr{} E(Z_{\text{red}})).
$$ 
The formal completion $\hat{E}$ is a sheaf on the big Zariski site again (cf.~\cite[II.1]{Artin-Mazur}).

Examples are $\hat{\mathbb{G}}_a$ and $\hat{\mathbb{G}}_m$; in both cases the 
sheaves are pro-represented by the formal scheme ${\rm Spf}(\OO_X[|t|])$, the group law
being $m^*(t)=t\hat{\otimes} 1+1\hat{\otimes} t$ for $\hat{\mathbb{G}}_a$ and 
$m^*(t)=t\hat{\otimes} 1+1\hat{\otimes} t-t\hat{\otimes} t$ for $\hat{\mathbb{G}}_m$. 

\subsection{} \label{subsubsection-deformation-cohomology}
Let $f:X\xr{} S$ be a morphism of schemes. For every morphism $g:T\xr{} S$ we define
a sheaf $\tilde{E}_T$ on the small Zariski site of $X_T:=X\times_S T$ as follows: let $\iota:X\times_S T_{{\rm red}}\xr{} X\times_S T$ be the base change of 
$T_{{\rm red}}\xr{} T$, we set 
$$
\tilde{E}_T:=\ker(E\xr{} \iota_*E),
$$
i.e.~$\tilde{E}_T(U)=\ker(E(U)\xr{} E(U\times_T T_{{\rm red}}))$ for every open $U\subset X_T$. 
By denoting $f_T:X_T\xr{} T$ the base change of $f$, we obtain for every 
integer $q\geq 0$ the sheaf
$
R^qf_{T*}\tilde{E}_T
$
on $T$.

\begin{definition}\label{definition-AM-functor}
For a morphism $g:T\xr{} S$ we set 
$$
\Phi^q(X/S,E)(g:T\xr{} S):=\Gamma(T,R^qf_{T*}\tilde{E}_T).
$$
\end{definition}

The assignment $\Phi^q(X/S,E)$ defines a sheaf on the big Zariski site of $S$, which 
is called the \emph{deformation cohomology} of $X/S$ with coefficients in $E$ 
\cite[II.(1.4)]{Artin-Mazur}. Indeed,
for $T'\xr{h} T \xr{g} S$, the natural map 
$(id_X\times h)^{-1}\tilde{E}_T\xr{} \tilde{E}_{T'}$ induces 
$$
h^{-1}R^qf_{T*}\tilde{E}_T \xr{} R^qf_{T'*}(id_X\times h)^{-1}\tilde{E}_T\xr{} R^qf_{T'*}\tilde{E}_{T'},
$$
equipping $\Phi^q(X/S,E)$ with the structure of a presheaf. Since
$$
\Gamma(U,R^qf_{T*}\tilde{E}_T)=\Gamma(U,R^qf_{U*}\left(\tilde{E}_T\mid_{f_T^{-1}U}\right))=\Gamma(U,R^qf_{U*}\tilde{E}_U),
$$
for every open $U$ of $T$, it is a sheaf. 

 For a morphism $h:Y\xr{} X$ over $S$ we get an induced morphism of sheaves
$$
\Phi^q(h,E): \Phi^q(X/S,E) \xr{} \Phi^q(Y/S,E),
$$
which is functorial in $h$.

Obviously, the deformation cohomology for $E$ and its formal completion $\hat{E}$ agree:
$$
\Phi^q(X/S,E)=\Phi^q(X/S,\hat{E})\qquad \text{for all $q$.}
$$

\begin{remark}\label{remark-etale-vs-Zariski}
  Artin and Mazur use the \'etale topology in \cite[II]{Artin-Mazur} for the definition
of the deformation cohomology. For the coefficients we will be considering \'etale
and Zariski deformation cohomology agree (Lemma \ref{lemma-quotients-coherent}).
\end{remark}

\subsection{}
If $E$ is formally smooth then, by definition, we have an exact sequence 
$$
0\xr{} \tilde{E}_T\xr{} E_{\mid X_T} \xr{} \iota_*(E_{\mid X_{T_{{\rm red}}}})\xr{} 0.
$$ 
Therefore we obtain a long exact sequence 
\begin{align}\label{equation-explanation-deformation-theory}
  \dots &\xr{} \Phi^{q}(X/S,E)_{\mid T}  \xr{} R^qf_{T*}(E_{\mid X_T}) \xr{} R^qf_{T_{{\rm red}}*}(E_{\mid X_{T_{{\rm red}}}}) \xr{} \\
  &\xr{} \Phi^{q+1}(X/S,E)_{\mid T}, \nonumber
\end{align}
and $\Phi^{*}(X/S,E)$ controls the deformation of cohomology classes on $T_{{\rm red}}$ to classes on $T$. A classical case 
is $E=\mathbb{G}_m$ and $S=\Spec(R)$ with $R$ a discrete valuation ring with uniformizer $t$ and residue field $k$. Let 
$T_n=\Spec(R/t^n)$; we get an exact sequence:
$$
\Phi^1(X/S,\hat{\mathbb{G}}_m)(R/t^n) \xr{} {\rm Pic}(X_{T_n}) \xr{} {\rm Pic}(X\otimes_{R}k) \xr{} \Phi^2(X/S,\hat{\mathbb{G}}_m)(R/t^n).
$$ 
For $L\in {\rm Pic}(X\otimes_R k)$ we obtain an obstruction class in $\varprojlim_{n}\Phi^2(X/S,\hat{\mathbb{G}}_m)(R/t^n)$ which 
becomes trivial if $L$ lifts to $\varprojlim_n {\rm Pic}(X_{T_n})$ .

\subsection{} 
Let $f:X\xr{} S$ be a flat and proper morphism. Assume that $f$ is  cohomologically flat in dimension zero. Then ${\rm Pic}\;X/S$ is 
represented by an algebraic space whose formal completion along the zero section is $\Phi^1(X/S,\hat{\mathbb{G}}_m)$ (cf. \cite[II.4]{Artin-Mazur}).

\subsection{}\label{subsubsection-formal-Lie-group}
If $E$ is a commutative formal group over $X$ then we say that $E$ is a 
\emph{commutative formal Lie group} if every point $x\in X$ admits an open affine neighborhood 
$U=\Spec(A)$ and an isomorphism $E\times_X U\cong {\rm Spf}(A[|x_1,\dots,x_d|])$ as formal 
schemes over $\Spec(A)$ such that the zero $0\in (E\times_X U)(\Spec(A))$ 
identifies with the morphism  $\Spec(A)\xr{}{\rm Spf}(A[|x_1,\dots,x_d|]))$ given by $x_i\mapsto 0$ for all $i$. 

For a commutative formal Lie group $E$ we define the \emph{tangent space} $TE$ by
$$
TE(U):=\ker(E(\Spec(\OO_U[\epsilon]/\epsilon^2))\xr{} E(U))
$$
for every open $U\subset X$. The tangent space is a locally free $\OO_X$-module 
of finite rank. For example, $T\hat{\mathbb{G}}_m=\OO_X=T\hat{\mathbb{G}}_a$.

\subsection{}
Again, let $f:X\xr{} S$ be a morphism and $E$ a commutative formal Lie group over $X$. Let
$g:T\xr{} S$ be a morphism and suppose that $T=\Spec(R)$ for a noetherian ring $R$. Set $T_i=\Spec(R/{\rm nil}(R)^i)$ for all $i\geq 1$, and 
$$
\tilde{E}^i_T:=\ker(E\xr{} \iota_{i*}E),
$$
with $\iota_i:X_{T_i}\xr{} X_T$ the base change of $T_i\xr{} T$.
For an open $U\subset X_T$, we have $\tilde{E}^i_T(U)=\ker(E(U)\xr{} E(U\times_T T_i))$.
The sheaf $\tilde{E}_T$ admits the filtration 
$$
\tilde{E}_T=\tilde{E}^1_T\supset \tilde{E}^2_T\supset \dots \supset \tilde{E}^n_T=0,
$$
provided that ${\rm nil}(R)^n=0$.

\begin{lemma}\label{lemma-quotients-coherent}
  For all $i\geq 1$, the quotient $\tilde{E}^i_T/\tilde{E}^{i+1}_T$ is a coherent 
$\OO_{X_{T_{{\rm red}}}}$-module. More precisely,
$$
({\rm nil}(R)^i\cdot \OO_{X_T}/{\rm nil}(R)^{i+1}\cdot \OO_{X_T})\otimes_{\OO_{X_T}} (id\times g)^*TE 
\xr{\cong}  \tilde{E}^i_T/\tilde{E}^{i+1}_T.
$$
\begin{proof}
  We may assume that $S=T$. There is a natural injective morphism
$$
\tilde{E}^i_S/\tilde{E}^{i+1}_S\xr{} \ker( \iota_{i+1*}E\xr{}  \iota_{i*}E).
$$
Locally one easily shows that the morphism is also surjective. We have a morphism
\begin{align*}
({\rm nil}(R)^i\cdot \OO_{X}/{\rm nil}(R)^{i+1}\cdot \OO_{X})\otimes_{\OO_{X}} TE &\xr{} \ker( \iota_{i+1*}E\xr{}  \iota_{i*}E) \\
n\otimes s &\mapsto \phi_n^*(s),  
\end{align*}
where $\phi_n:X\times_S S_{i+1}\xr{} \Spec(\OO_X[\epsilon]/\epsilon^2)$ is 
defined over $X$ by $\phi_n^*(\epsilon)=n$. Again, one checks locally that 
the morphism is an isomorphism.
\end{proof}
\end{lemma}

\section{Representation by a formal group: reduction to nilpotent algebras}
\subsection{}
We will be mainly concerned with the following question when $S$ is flat over $\Spec(\Z)$.
\begin{question}\label{question-main} 
  Let $E$ be a  commutative formal Lie group over $X$. When is $\Phi^q(X/S,E)$
pro-representable by a  commutative formal Lie group over $S$? 
\end{question}

For simplicity we will work with noetherian schemes.

A pro-representing formal scheme $\mathfrak{X}$ with an isomorphism 
$\mathfrak{X}\cong \Phi^q(X/S,E)$ is unique up to unique isomorphism. Indeed,
if $J\subset \OO_{\mathfrak{X}}$ is the maximal ideal of definition then the 
scheme $X_n:=(\mathfrak{X},\OO_{\mathfrak{X}}/J^n)$ represents the restriction of 
$\Phi^q(X/S,E)$ to the category of schemes $T$ over $S$ such that ${\rm nil}(\OO_T)^n=0$. Thus $X_n$ is unique, which implies that $\mathfrak{X}=\varinjlim_n X_n$ is unique. In particular, Question \ref{question-main} is local in $S$.

\subsection{}
Let $R$ be a commutative ring. For an $R$-algebra $A$ we define a new $R$-algebra 
$N_R(A)$ by
$$
N_R(A):= R\oplus {\rm nil}(A)/\{(r,-r); r\in {\rm nil}(R)\}.
$$
The multiplication is given by $(r,n)\cdot (r',n')=(rr',rn'+r'n+nn')$. Obviously,
$$
{\rm nil}(A)\xr{\cong} {\rm nil}(N_R(A)), \quad n\mapsto (0,n),  
$$
is an isomorphism. Clearly, $N_R$ defines a functor 
$$
N_R: \text{($R$-algebras)} \xr{} \text{($R$-algebras)}.
$$
Moreover, there is a morphism of functors $N_R \xr{} id$ induced by the natural 
map
\begin{equation}\label{equation-NRA-to-A}
N_R(A)\xr{} A, \quad (r,n)\mapsto r+n.  
\end{equation}
Let $\mc{C}$ be the full subcategory consisting of $R$-algebras $A$ such that 
(\ref{equation-NRA-to-A}) is an isomorphism. Then $N_R$ takes values in $\mc{C}$
and is right adjoined to the forgetful functor $\mc{C}\xr{} \text{($R$-algebras)}$:
$$
\Hom_R(B,A)=\Hom_R(B,N_R(A)) \quad \text{for all $B\in \mc{C}$ and any $R$-algebra $A$.}
$$

\begin{proposition}\label{proposition-A-to-NRA}
Let $f:X\xr{} S=\Spec(R)$ be a flat morphism between 
noetherian schemes. Let $E$ be a commutative formal Lie group over $X$ and let 
$g:\Spec(A)\xr{} S$ be a morphism. We assume that $A$ is noetherian. For all 
$q\geq 0$ there is a functorial isomorphism 
\begin{equation}\label{equation-functorial-map-HqNRA-HqA}
H^q(X_{\Spec(N_R(A))},\tilde{E}_{\Spec(N_R(A))})  \xr{\cong} H^q(X_{\Spec(A)},\tilde{E}_{\Spec(A)}). 
  \end{equation}
\begin{proof}
  By using (\ref{equation-NRA-to-A}) we obtain a functorial morphism (\ref{equation-functorial-map-HqNRA-HqA}). Set $T:=\Spec(A)$; for every open $U\subset X$ we 
claim that 
\begin{equation}\label{equation-functorial-map-H0NRA-H0A}
  \Gamma(U_{\Spec(N_R(A))},\tilde{E}_{\Spec(N_R(A))}) \xr{} \Gamma(U_{T},\tilde{E}_{T})
\end{equation}
is an isomorphism. 

In fact, we may assume that $U=X$. Suppose first that $X$ is separated. Then 
it is sufficient to prove (\ref{equation-functorial-map-H0NRA-H0A}) for every
affine open $U=\Spec(B)$ in $X$. Moreover, we may assume that 
$E\times_XU\cong {\rm Spf}(B[|x_1,\dots,x_d|])$. Then (\ref{equation-functorial-map-H0NRA-H0A}) reads
$$
\prod_{i=1}^d \ker(B\otimes_R N_R(A)\xr{} B\otimes_R (N_R(A)/{\rm nil}(A)))
\xr{} \prod_{i=1}^d \ker(B\otimes_R A\xr{} B\otimes_R (A/{\rm nil}(A))).
$$
Since $B$ is flat over $R$ both sides equal 
$
(B\otimes_R {\rm nil}(A))^d.
$

If $X$ is not separated then we still have (\ref{equation-functorial-map-H0NRA-H0A}) 
for every separated open $U$ of $X$. Since an open of a separated scheme
is separated again, we can simply use a covering of $X$ by separated (e.g.~affine) open sets to conclude (\ref{equation-functorial-map-H0NRA-H0A}) for $X$.

Let us prove that (\ref{equation-functorial-map-HqNRA-HqA}) is an isomorphism. 
Again, suppose first that $X$ is separated. Since (\ref{equation-functorial-map-H0NRA-H0A}) 
holds it is sufficient to show that $H^q(X_{T},\tilde{E}_{T})$, for any noetherian 
affine scheme $T$, can be calculated by the cohomology of the \v{C}ech complex
associated to an affine open covering. This follows because $\tilde{E}_T$ has 
a finite filtration such that the graded pieces are coherent (Lemma~\ref{lemma-quotients-coherent}). 

If $X$ is not separated then we take a finite covering $\mathfrak{U}$ of $X$ by affine open sets.
By using the spectral sequence 
$$
E_1^{p,q}=\bigoplus_{i_0<\dots<i_p} H^q(\bigcap_{j=0}^p U_{i_{j},T},\tilde{E}_T) \Rightarrow H^{p+q}(X_T,\tilde{E}_T), 
$$
we reduce to the statement for the separated schemes $\bigcap_{j=0}^p U_{i_{j}}$.  
\end{proof}
\end{proposition}

\subsection{}
Let $R$ be a commutative algebra. By a \emph{nilpotent} $R$-\emph{algebra} $\mc{N}$  
we mean an $R$-algebra $\mc{N}$  without $1$-element, such that $\mc{N}^r=0$ for an integer $r\geq 1$. There is an obvious $R$-algebra $R\oplus \mc{N}$ attached 
to every nilpotent $R$-algebra $\mc{N}$.

Let $f:X\xr{} \Spec(R)$ be a morphism and $E$ a commutative formal Lie group over $X$. 
For a nilpotent $R$-algebra $\mc{N}$  
we define a sheaf $\tilde{E}_{\mc{N}}$ on the small Zariski site of 
$X_{\Spec(R\oplus \mc{N})}$ by 
\begin{align*}
  \tilde{E}_{\mc{N}} 
                   &=\ker(\tilde{E}_{\Spec(R\oplus \mc{N})} \xr{}  \iota_*\tilde{E}_{\Spec(R)}),
\end{align*}
where $\iota:\Spec(R)\xr{} \Spec(R\oplus \mc{N})$ is induced by $R\oplus \mc{N}\xr{} R, (r,n)\mapsto r$. We note that $X\xr{} X_{\Spec(R\oplus \mc{N})}$ is topologically 
an isomorphism, and we can consider $\tilde{E}_{\mc{N}}$ as sheaf on $X$.

For an integer $q\geq 0$ we define 
\begin{align} \label{align-def-psi}
\Psi^q(X/S,E):\text{(nilpotent $R$-algebras)} &\xr{} \text{(abelian groups)} \\
  \Psi^q(X/S,E)(\mc{N}) &= H^q(X,\tilde{E}_{\mc{N}}). \nonumber
\end{align}


\begin{notation}\label{notation-representable-nilpotent-algebras}
We say that a functor 
$$
  \Psi: \text{(nilpotent $R$-algebras)} \xr{} \text{(abelian groups)} \\
$$ 
is pro-represented by the formal group $\mathfrak{X}$ over $R$ if 
$$
\Psi \cong \left[ \mc{N} \mapsto \ker(\mathfrak{X}(R\oplus \mc{N}) \xr{} \mathfrak{X}(R)) \right].
$$  
\end{notation}

\begin{proposition}\label{proposition-psi-phi}
  Let $q\geq 0$ be an integer. Let $f:X\xr{} S=\Spec(R)$ be a flat morphism between noetherian schemes. Let $E$ be a commutative formal Lie group over $X$.
If $\Psi^q(X/S,E)$ is pro-represented by a commutative formal Lie group $\mathfrak{X}$ (over $S$) then 
the restriction of $\Phi^q(X/S,E)$ to the category of noetherian schemes (over $S$) is pro-represented by $\mathfrak{X}$.
\begin{proof}
  For a noetherian $R$-algebra $A$ we have functorial isomorphisms 
$$
\Psi^q(X/S,E)({\rm nil}(A))\xr{\cong} H^q(X_{\Spec(N_R(A))},\tilde{E}) \xr{\cong} H^q(X_{\Spec(A)},\tilde{E}). 
$$ 
The first isomorphism is obvious and the second follows from Proposition \ref{proposition-A-to-NRA}. On the other hand, 
$\ker(\mathfrak{X}(R\oplus {\rm nil}(A)) \xr{} \mathfrak{X}(R))\cong \mathfrak{X}(A)$. Therefore the functor 
$$
\text{($R$-alg)}\xr{} \text{(abelian groups)}, \quad  A\mapsto H^q(X_{\Spec(A)},\tilde{E}_{\Spec(A)}),
$$
is pro-represented by $\mathfrak{X}$. 

We have a morphism of functors
\begin{equation}\label{equation-psi-to-phi}
[A\mapsto H^q(X_{\Spec(A)},\tilde{E})] \xr{} [A\mapsto \Phi^q(X/S,E)(\Spec(A))],  
\end{equation}
and the right hand side is the sheafication of the left hand side. Since 
$\mathfrak{X}$ defines a sheaf, the restriction of $\Phi^q(X/S,E)$ to the 
category of noetherian affine schemes is represented by $\mathfrak{X}$. Again, 
since $\Phi^q(X/S,E)$ and $\mathfrak{X}$ are both sheaves we obtain 
$\Phi^q(X/S,E) \cong \mathfrak{X}$.
\end{proof}
\end{proposition}

As a criterion for pro-representability by a formal Lie group we will use the
following proposition.   

\begin{proposition}\cite[2.32]{Zink}\label{proposition-exact-representable}
  Let $H:\text{(nilpotent $R$-algebras)} \xr{} \text{(ab. groups)}$ be an 
exact functor which commutes with arbitrary direct sums. Suppose that $H(xR[x]/x^2R[x])$ is a free and finitely generated $R$-module. Then $H$ is represented 
by a commutative formal Lie group $\mathfrak{X}$ in the sense of Notation \ref{notation-representable-nilpotent-algebras}.
\end{proposition}

\begin{remark}\label{remark-exact-commutes-with-direct-sums}
  A short exact sequence in the category of nilpotent $R$-algebras is by definition  a sequence $0\xr{} N_1\xr{} N_2\xr{} N_3\xr{} 0$ of morphisms of 
nilpotent $R$-algebras that is exact as a sequence of $R$-modules. A functor $H$
as above is called exact if it preserves exact sequences. Left and right 
exact functors are defined in the same way.

By definition, a functor $H$ commutes with arbitrary direct sums if the natural map
$$
\bigoplus_{i\in I} H(N_I) \xr{} H(\bigoplus_{i\in I} N_i)
$$ 
is an isomorphism for all sets $\{N_i\}_{i\in I}$ 
of nilpotent $R$-algebras with the property that there exists $k\geq 1$ with $N_i^k=0$ for all $i\in I$.
\end{remark}

\section{Representation by formal groups}

\subsection{Formulation of the main theorem}
Let $f:X\xr{} S$ be a morphism of schemes. Let $E$ be a commutative formal Lie group over $X$ (see \ref{subsubsection-formal-Lie-group}), e.g. $E=\hat{\mathbb{G}}_m$, the 
formal completion of $\mathbb{G}_m$ (see \ref{subsubsection-formal-completion}).
The tangent space of $E$ is denoted by $TE$.
Recall Definition \ref{definition-AM-functor} for the Artin-Mazur functor $\Phi^q(X/S,E)$.

\begin{thm}\label{thm-main}
  Let $f:X\xr{} S$ be a flat morphism between noetherian schemes. 
Suppose that $S$ is flat over $\Spec(\Z)$. Let $E$ be a commutative formal Lie group over $X$.
Suppose that $R^qf_*TE$ is locally free and of finite rank for all $q\geq q_0$. 
Then the restriction of $\Phi^q(X/S,E)$ to the category of noetherian schemes (over $S$) is pro-representable by a commutative formal Lie group  with tangent space $R^qf_*TE$ for all $q\geq q_0$.
\end{thm}

\subsection{Proof of the main theorem}

The following lemma on which the proof is based, was shown to us by the referee for International Mathematics Research Notices.

\begin{lemma}\label{lemma-referee}
  Let $R$ be a $\Z$-torsion free ring. Let 
$$\Phi:\text{(nilpotent $R$-algebras)} \xr{} \text{(abelian groups)}$$
be a right exact functor. Assume that the restriction of $\Phi$ to the following two categories is exact:
\begin{enumerate}
\item (nilpotent $R\otimes \Q$-algebras),
\item (nilpotent $R$-algebras $N$ such that $N^2=0$).
\end{enumerate}
Then $\Phi$ is exact.
\begin{proof}
It suffices to show that $\Phi$ preserves injective maps of nilpotent $R$-algebras. We remark that an exact functor preserves injective maps \cite[Theorem~2.16]{Zink}.

\emph{Step 1.} We claim that $\Phi(N)\xr{} \Phi(N\otimes_{\Z} \Q)$ is injective
if $N$ is $\Z$-torsion-free. By assumption the claim holds if $N^2=0$. 
For the general case, set $N_{\Q}:=N\otimes_{\Z} \Q$ and define 
$N':=N^2_{\Q}\cap N$. Note that $N'\subset N$ is an ideal and $N/N'$ is $\Z$-torsion-free. We obtain a diagram of exact sequences
$$
\xymatrix{
&
\Phi(N')\ar[r]\ar[d]
&
\Phi(N)\ar[r]\ar[d]
&
\Phi(N/N'') \ar[r]\ar[d]
&
0
\\
0\ar[r]
&
\Phi(N'\otimes_{\Z}\Q)\ar[r]
&
\Phi(N\otimes_{\Z}\Q)\ar[r]
&
\Phi(N/N''\otimes_{\Z}\Q) \ar[r]
&
0.
}
$$ 
The left and right vertical arrows are injective by induction on the exponent that annihilates the nilpotent algebra. Hence the central vertical arrow is injective.

\emph{Step 2.} If $N'\xr{} N$ is an injective morphism of nilpotent $R$-algebras and $N'$ is $\Z$-torsion-free then
$\Phi(N')\xr{} \Phi(N)$ is injective. Indeed, the first arrow in $$\Phi(N')\xr{} \Phi(N'\otimes_{\Z}\Q)\xr{} \Phi(N\otimes_{\Z} \Q)$$ is injective by using the first step, and the second arrow is injective by assumption. 

\emph{Step 3.} Let $N'\xr{} N$ be any injective morphism of nilpotent $R$-algebras.  There exists a surjective morphism of 
nilpotent $R$-algebras $T\xr{} N$ such that $T$ is a free $R$-module. Let 
$T':=N'\times_{N} T$ be the fiber product, and let $I=\ker(T\xr{} N)$. 
Since $R$ is $\Z$-torsion-free, $T'$ and $I$ are $\Z$-torsion-free.
By using the second step we obtain a morphism of exact sequences
$$
\xymatrix
{
\Phi(I) \ar[r]\ar[d]^{=}
&
\Phi(T')\ar[r]\arir[d]
&
\Phi(N')\ar[r]\ar[d]
&
0
\\
\Phi(I)\ar[r]
&
\Phi(T)\ar[r]
&
\Phi(N)\ar[r]
&
0,
}
$$ 
where the central vertical arrow is injective. This implies that $\Phi(N')\xr{} \Phi(N)$ is injective. 
\end{proof}
\end{lemma}

\begin{proof}[Proof of Theorem \ref{thm-main}]
In view of the uniqueness of a pro-representing formal group we may assume that $S=\Spec(R)$ is affine. By Proposition \ref{proposition-psi-phi} it suffices to show that $\Psi^q(X/S,E)$
is pro-representable by a commutative formal Lie group (see (\ref{align-def-psi}) for the definition of $\Psi^q(X/S,E)$).  

Suppose $N$ is a nilpotent $R$-algebra such that $N^2=0$. Then
$$
\Psi^q(X/S,E)(N)=H^q(X,TE\otimes_R N),
$$
and our assumptions imply $H^q(X,TE\otimes_R N)=H^q(X,TE)\otimes_R N$ for all 
$q\geq q_0$. In particular, the restriction of $\Psi^q(X/S,E)$ to the category
of nilpotent $R$-algebras with vanishing square is exact if $q\geq q_0$.

Let us denote by $S_{\Q}=S\times_{\Spec(\Z)}\Spec(\Q)$ the base change. For any 
nilpotent $R\otimes_{\Z} \Q$-algebra $N$ we have
$$
\Psi^q(X/S,E)(N)=\Psi^q(X_{\Q}/S_{\Q},E_{\Q})(N)\cong H^q(X_{\Q},TE_{\Q}\otimes_{R\otimes \Q}N),
$$ 
because there is an isomorphism $E_{\Q}\cong \widehat{TE}_{\Q}$ (see \ref{subsection-logarithm}). As above 
we conclude that the restriction of $\Psi^q(X/S,E)$ to the category
of nilpotent $R\otimes \Q$-algebras is exact if $q\geq q_0$.

In view of Lemma \ref{lemma-referee} we need to show that $\Psi^q(X/S,E)$, for $q\geq q_0$, is right exact in order to conclude that $\Psi^q(X/S,E)$ is exact. We will prove this by descending induction on $q$. Indeed, since $X/S$ is flat, every short exact sequence $0\xr{} N_1\xr{} N_2\xr{} N_3\xr{} 0$ of nilpotent $R$-algebras gives rise to a long exact sequence: 
\begin{multline*}
\dots \xr{} \Psi^q(X/S,E)(N_1) \xr{} \Psi^q(X/S,E)(N_2) \xr{} \Psi^q(X/S,E)(N_3)\xr{} \\ \Psi^{q+1}(X/S,E)(N_1)\xr{} \dots,  
\end{multline*}
and we know $\Psi^q(X/S,E)=0$ for $q>\dim X$.

Finally, in order to apply Proposition \ref{proposition-exact-representable}, we have to show that $\Psi^q(X/S,E)$, for $q\geq q_0$, commutes with arbitrary direct sums. Since $\Psi^q(X/S,E)$ is exact it is sufficient to consider the 
restriction to nilpotent $R$-algebras $N$ with $N^2=0$, where the statement is
evident.
\end{proof}

\section{Deformation cohomology of $\mathbb{G}_m$ and Witt vector cohomology}

\subsection{Cartier modules}
Let $R$ be a commutative ring. For a functor
\begin{equation}\label{equation-functor-nil-to-ab}
F: \text{(nilpotent $R$-algebras)} \xr{} \text{(abelian groups)},  
\end{equation}
we set
$$
F(xR[|x|]):=\varprojlim_n F(xR[x]/x^nR[x]).
$$
The assignment $C:F\mapsto F(xR[|x|])$ is functorial. For a morphism $\Phi:F\xr{} G$ of functors 
we denote by $\Phi_{R[|x|]}:F(xR[|x|]) \xr{}  G(xR[|x|])$ the induced morphism.

Define the functor $\Lambda$ by  
\begin{align*}
\Lambda:&\text{(nilpotent $R$-algebras)} \xr{} \text{(abelian groups)},\\
\Lambda(\mc{N})&:=\{1+n_1t+n_2t^2+\dots+n_rt^r\mid r\geq 0, n_i\in \mc{N}\}\subset (R\oplus \mc{N})[t]^{\times}.  
\end{align*}
Obviously, $\Lambda$ is exact.

\begin{thm}\cite[3.5]{Zink}\label{thm-1-CT}
Let $H$ be a left exact functor. The morphism of abelian groups 
$$
\Hom(\Lambda,H)\xr{\cong} H(xR[|x|]), \quad \Phi\mapsto \Phi_{R[|x|]}(1-xt),
$$ 
is an isomorphism, where $1-xt\in \Lambda(xR[|x|])$.
\end{thm}
In the theorem, $\Hom(\Lambda,H)$ denotes the natural transformations from 
$\Lambda$ to $H$.

\begin{definition}[Cartier ring] \label{definition-Cartierring}
  Let $R$ be a commutative ring. We set 
$$
\mathbb{E}_R:={\rm End}(\Lambda)^{op},
$$
where $(.)^{op}$ denotes the opposite ring; $\mathbb{E}_R$ is called the \emph{Cartier ring}. 
\end{definition}

Theorem \ref{thm-1-CT} implies immediately that $H(xR[|x|])$ comes equipped 
with a natural left $\mathbb{E}_R$-module structure; $H(xR[|x|])$ is called the 
\emph{Cartier module} of $H$.

In view of  
$$
\mathbb{E}^{op}_R=\End(\Lambda)\xr{\cong, \lambda} \Lambda(xR[|x|]),
$$
we obtain the following elements in $\mathbb{E}_R$:
\begin{align}
  V_n:=\lambda^{-1}(1-x^nt), \quad F_n&:=\lambda^{-1}(1-xt^n), &&\text{for $n\geq 1$,}\\
  [c]&:=\lambda^{-1}(1-cxt), \quad &&\text{for $c\in R$.}
\end{align}

\begin{thm}\cite[3.12]{Zink}
  Every element $\xi\in \mathbb{E}_R$ has a unique representation 
$$
\xi=\sum_{n,m\geq 1} V_n[a_{n,m}]F_m,
$$
with $a_{n,m}\in R$, and for every fixed $n$ almost all $a_{n,m}$ vanish; in other words,
$a_{n,m}=0$ for $m\geq m_0(n)$ with $m_0(n)$ depending on $n$. 
\end{thm}

\begin{proposition}\cite[3.13]{Zink}
  The following relations hold in $\mathbb{E}_R$:
  \begin{align*}
    &F_1=1=V_1 &&F_nV_n=n &\text{for $n\geq 1$,} \\
    &[c]V_n=V_n[c^n]  &&F_n[c]=[c^n]F_n &\text{for $n\geq 1$, $c\in R$,}\\
    &V_mV_n=V_{n\cdot m} &&F_nF_m=F_{n\cdot m} &\text{for $m,n\geq 1$,}\\
    & &&F_nV_m=V_mF_n &\text{if ${\rm gcd}(m,n)=1$.}
  \end{align*}
For $c_1,c_2\in R$ we have
\begin{align*}
  [c_1][c_2]&=[c_1c_2]\\
  [c_1+c_2]&=[c_1]+[c_2]+\sum_{n=2}^{\infty}V_n[a_n(c_1,c_2)]F_n, 
\end{align*}
with $a_n\in \Z[X_1,X_2]$.
\end{proposition}

The subset $\{\sum_{n\geq 1} V_n[c_n]F_n\mid \forall n: c_n\in R\}$ of $\mathbb{E}_R$ 
defines a subring $\mathbb{W}(R)$ and is called the ring of \emph{big Witt vectors} of $R$. 

\begin{example}
We have an isomorphism of $\mathbb{E}_{R}$-modules
\begin{align*}
\mathbb{E}_{R}/ \mathbb{E}_{R}\cdot (F_n-1\mid n\in \Z_{\geq 2}) &\xr{} 
\hat{\mathbb{G}}_m(xR[|x|])\\
1\mapsto 1-x, \quad V_n\mapsto 1-x^n,& \quad  [c]\mapsto 1-cx.
\end{align*}
In other words, the Cartier module of $\hat{\mathbb{G}}_{m,R}$ is $\mathbb{E}_{R}$ modulo the (left) ideal 
generated by the elements $F_n-1$ where $n$ runs through all natural numbers. 
Via 
$$
\mathbb{W}(R)\xr{} \mathbb{E}_{R}/ \mathbb{E}_{R}\cdot (F_n-1\mid n\in \Z_{\geq 2}) \xr{} \hat{\mathbb{G}}_m(xR[|x|])
$$
we obtain an isomorphism $\mathbb{W}(R)\cong \hat{\mathbb{G}}_m(xR[|x|])$ of 
$\mathbb{W}(R)$-modules. 

We define 
$$
\mathbb{W}_{[1,n]}(R):=\hat{\mathbb{G}}_m(xR[x]/x^{n+1}R[x])=\{1+\sum_{i=1}^n b_ix^i\mid b_i\in R\}.
$$

For the Cartier module of $\hat{\mathbb{G}}_a$ we have an isomorphism 
\begin{align*}
\mathbb{E}_{R}/ \mathbb{E}_{R}\cdot (F_n\mid n\in \Z_{\geq 2})&\xr{} \hat{\mathbb{G}}_a(xR[|x|])\\
1\mapsto x, \quad V_n\mapsto x^n,& \quad [c]\mapsto cx. 
\end{align*}
More explicitly, with the identification
$
\hat{\mathbb{G}}_a(xR[|x|])=xR[|x|],
$
we get 
\begin{equation}\label{equation-Fn-Vn-Ga}
  V_n(g(x))=g(x^n), \quad F_n(cx^i)=\begin{cases} ncx^{\frac{i}{n}} &\text{if $n$ divides $i$,} \\ 0 &\text{otherwise,} \end{cases}\quad [c](g(x))=g(cx).
\end{equation}
\end{example}

\begin{proposition}\label{proposition-new-Cartier-module-Wittvector-cohomology}
  Let $f:X\xr{} \Spec(R)$ be a flat separated morphism between noetherian schemes.
We have 
$$
\Psi^{q}(X/S,\hat{\mathbb{G}}_m)(xR[|x|])=\varprojlim_{n} H^{q}(X,\mathbb{W}_{[1,n]}\OO_X).
$$
\begin{proof}
We need to show $ \Psi^{q}(X/S,\hat{\mathbb{G}}_m)(xR[x]/x^{n+1}R[x])=H^{q}(X,\mathbb{W}_{[1,n]}\OO_X)$. 

Let $\{U_i\}_{i=1,\dots,n}$ be an affine covering of $X$. We set 
$A_{J}=\Gamma(\bigcap_{k=0}^s U_{\jmath_k}, \OO_X)$ for all $J=(\jmath_0<\dots <\jmath_s)$. Set $E=\hat{\mathbb{G}}_m$; if $\mc{N}$ is a nilpotent $R$-algebra then 
Lemma \ref{lemma-quotients-coherent} implies that 
$$
H^q(X,\tilde{E}_{\mc{N}})=H^q(\mc{C}^*(\{U_i\},\tilde{E}_{\mc{N}})),
$$
where $\mc{C}^*(\{U_i\},\tilde{E}_{\mc{N}})$ is the \v{C}ech complex. Explicitly, we have 
$$
\mc{C}^p(\{U_i\},\tilde{E}_{xR[x]/x^{n+1}R[x]})= 
\bigoplus_{|J|=p+1} \mathbb{W}_{[1,n]}(A_J).
$$ 
Since the cohomology of $\mathbb{W}_{[1,n]}(\OO_X)$ is also computed by using the \v{C}ech complex, we are done.
\end{proof}
\end{proposition}

\begin{proposition}\label{proposition-Cartier-module-Wittvector-cohomology}
  Let $f:X\xr{} S$ be a flat separated morphism between noetherian schemes. 
Suppose that $S=\Spec(R)$ is flat over $\Spec(\Z)$. 
Suppose that $R^qf_*\OO_X$ is locally free and of finite rank for all 
$q\geq q_0$. 
The Cartier module of $\Phi^{q_0}(X/S,\hat{\mathbb{G}}_m)$ is given 
by $H^{q_0}(X,\mathbb{W}\OO_X)$ where $\mathbb{W}\OO_X$ is the sheaf of big Witt vectors.
\begin{proof}
Proposition \ref{proposition-new-Cartier-module-Wittvector-cohomology} implies that 
we have to show $$H^{q_0}(X,\mathbb{W}\OO_X)=\varprojlim_{n} H^{q_0}(X,\mathbb{W}_{[1,n]}\OO_X).$$ This follows from 
$R^1\varprojlim_{n} H^{q_0-1}(X,\mathbb{W}_{[1,n]}\OO_X)=0$.  By using 
the long exact sequence for $\Psi^*(X/S,\hat{\mathbb{G}}_m)$, Theorem \ref{thm-main} implies that $\Psi^{q_0-1}(X/S,\hat{\mathbb{G}}_m)$ is right exact. 
Therefore the transition maps $H^{q_0-1}(X,\mathbb{W}_{[1,n+1]}\OO_X)\xr{} H^{q_0-1}(X,\mathbb{W}_{[1,n]}\OO_X)$ are surjective.
\end{proof}
\end{proposition}

\section{Gauss-Manin connection} \label{subsection-Gauss-Manin}

\subsection{}
Let $S$ be a smooth scheme  over $\Spec(\Z)$. Let $\mathfrak{X}$ be 
a commutative formal Lie group over $S$. We have the de-Rham complex 
$\Omega^*_{\mathfrak{X}/S}$ at disposal and the Katz-Oda \cite{Katz-Oda} construction yields
an integrable connection 
$$
\nabla:H^i(\Omega^*_{\mathfrak{X}/S})\xr{} 
H^i(\Omega^*_{\mathfrak{X}/S})\otimes_{\OO_S}\Omega^1_{S/\Z}\qquad \text{for all $i$.}
$$ 
Locally, we can write a closed $i$-form $\omega$ as
$$
\omega=\sum_{J=(j_1,\dots,j_i)}\omega_J\cdot dx_{j_1}\wedge \dots \wedge dx_{j_i},
$$ 
with $\omega_J=\sum_{K=(k_1,\dots,k_d)} a_{J,K}\cdot x_1^{k_1}\cdots x_{d}^{k_d}$, and
$a_{J,K}\in \OO_S$. For a section $\xi\in T_{S/\Z}$ the action is simply given by
$$
\nabla_{\xi}(\omega)=\sum_{J=(j_1,\dots,j_i)}\sum_{K=(k_1,\dots,k_d)} \xi(a_{J,K})\cdot x_1^{k_1}\cdots x_{d}^{k_d} \cdot dx_{j_1}\wedge \dots \wedge dx_{j_i}.
$$
In view of Proposition \ref{proposition-invariant-forms} we have a
morphism 
$$
\eta:T\mf{X}^{\vee}\xr{} \Omega^1_{\mf{X}/S},
$$
inducing an isomorphism with the  invariant $1$-forms. By using Lemma \ref{lemma-invariant-implies-closed}
we obtain 
\begin{equation}\label{equation-eta-bar}
  \bar{\eta}:T\mf{X}^{\vee}\xr{} H^1(\Omega^*_{\mf{X}/S}).
\end{equation}
\begin{definition}\label{definition-de-Rham-inv}
  We define 
$
H^1_{{\rm inv}}(\mf{X})\subset H^1(\Omega^*_{\mf{X}/S})
$
to be the smallest $\OO_S$-submodule such that $H^1_{{\rm inv}}(\mf{X})$
contains the image of $\bar{\eta}$, and the connection on $H^1(\Omega^*_{\mf{X}/S})$
induces a connection on $H^1_{{\rm inv}}(\mf{X})$.
\end{definition}

\subsection{}
It is more convenient to work with modules rather than with connections. We denote by $D_S$ the sheaf differential operators    
and we denote by $D'_S$ the subsheaf generated (as $\OO_S$-algebra) by differential 
operators of order $\leq 1$. Locally, we can write $S=\Spec(R)$ and suppose 
that there is an \'etale morphism $\Z[x_1,\dots,x_n]\xr{} R$. As an $R$-module
we have 
$$
D'_S(S)=R\otimes_\Z \Z[\partial_1,\dots,\partial_n]
$$
with $\partial_i(x_j)=\delta_{i,j}$, the ring structure being uniquely determined
by 
$$
\partial_i\partial_j=\partial_j\partial_i, \quad \partial_ir=r\partial_i+\partial_i(r) \quad \text{for all $i,j$ and $r\in R$.}
$$

We have an obvious equivalence of categories between $\OO_S$-modules with 
integrable connection and $D'_S$-modules. By definition, 
$H^1_{{\rm inv}}(\mathfrak{X})$ is the $D'_S$-submodule of $H^1(\Omega^*_{\mf{X}/S})$ 
generated by $\bar{\eta}(T\mf{X}^{\vee})$, thus it comes equipped with a 
surjective morphism 
$$
D'_S\otimes_{\OO_S} T\mf{X}^{\vee} \xr{} H^1_{{\rm inv}}(\mathfrak{X}).
$$
In particular, $H^1_{{\rm inv}}(\mathfrak{X})$ is a quasi-coherent $\OO_S$-module.

\subsection{}
Let $f:X\xr{} S$ be a smooth, projective morphism of relative dimension $r$.
As above we assume that $S$ is smooth over $\Spec(\Z)$. The de-Rham cohomology
$$
H^i_{{\rm dR}}(X/S):=R^if_*(\Omega^*_{X/S})
$$
comes equipped with the Gauss-Manin connection, hence defines a $D'_S$-module.

\begin{thm}\label{thm-GM-simple}
  Let $f:X\xr{} S$ be a smooth, projective morphism of relative dimension $r$.
Suppose that $S$ is smooth over $\Spec(\Z)$ and suppose that 
$R^jf_*\Omega^i_{X/S}$ is locally free for all $i,j$. 
Then $H^1_{{\rm inv}}(\Phi^i(X/S))$ is a subquotient of $H^{2r-i}_{{\rm dR}}(X/S)$
as $D'_S$-module for all $i\geq 0$.
\end{thm}

In fact, it will be easier to prove the theorem when stated in a 
more precise way (Theorem \ref{thm-GM-better}). Theorem \ref{thm-main}
guarantees the pro-representability of 
$\Phi^i(X/S)=\Phi^i(X/S,\hat{\mathbb{G}}_m)$. 
Let $X_{\Q}$ and $S_{\Q}$ denote the base change to $\Spec(\Q)$.  Obviously, 
we have
$$
H^{2r-i}_{{\rm dR}}(X/S)\otimes_{\OO_S}\OO_{S_{\Q}}=H^i_{{\rm dR}}(X_{\Q}/S_{\Q}).
$$
Hence, Theorem \ref{thm-GM-simple} implies the following corollary.

\begin{corollary}\label{corollary-GM-even-simpler}
Let $f:X\xr{} S$ be a smooth, projective morphism of relative dimension $r$.
Suppose that $S$ is smooth over $\Spec(\Z)$ and suppose that 
$R^jf_*\OO_X$ is locally free for all $j$. Then, for all $i$, the sheaf 
$H^1_{{\rm inv}}(\Phi^i(X/S))\otimes_{\OO_S}\OO_{S_{\Q}}$ is a subquotient of 
$H^{2r-i}_{{\rm dR}}(X_{\Q}/S_{\Q})$ as module over the sheaf of differential operators
of $S_{\Q}$.
\begin{proof}
We know that $R^jf_*\Omega^i_{X_{\Q}/S_{\Q}}$ is locally free by Hodge theory. Thus 
there exists an integer $n\neq 0$ such that the base change of $f$ to $S_{\Z[n^{-1}]}=S\times_{\Spec(\Z)}\Spec(\Z[n^{-1}])$ satisfies the assumptions in Theorem  \ref{thm-GM-simple}. Since 
$$
H^1_{{\rm inv}}(\Phi^i(X_{\Z[n^{-1}]}/S_{\Z[n^{-1}]}))\otimes_{\OO_{S_{\Z[n^{-1}]}}}\OO_{S_{\Q}}=H^1_{{\rm inv}}(\Phi^i(X/S))\otimes_{\OO_S}\OO_{S_{\Q}}, 
$$
the statement follows.
\end{proof}
\end{corollary}

Of course, $H^1_{{\rm inv}}(\Phi^i(X_{\Q}/S_{\Q}))$ vanishes. But $H^1_{{\rm inv}}(\Phi^i(X/S))\otimes_{\OO_{S}}\OO_{S_{\Q}}$ is in general non-trivial 
(see Examples \ref{example-GM}). 

Since $f$ is of relative dimension $r$, we get a natural morphism
$$
\epsilon:R^{r-i}f_*\Omega^{r}_{X/S}\xr{} H^{2r-i}_{{\rm dR}}(X/S).
$$
Therefore, $D'_S\cdot \epsilon(R^{r-i}f_*\Omega^{r}_{X/S})$ is a $D'_S$-submodule 
of $H^{2r-i}_{{\rm dR}}(X/S)$.
Moreover, Grothendieck duality implies the existence of a 
canonical morphism  
\begin{equation}\label{equation-duality-tau}
\tau:R^{r-i}f_*\Omega^{r}_{X/S}\xr{} (R^if_*\OO_X)^{\vee},  
\end{equation}
coming from the pairing 
\begin{equation}\label{equation-Grothendieck-duality}
R^{r-i}f_*\Omega^{r}_{X/S} \times R^if_*\OO_X \xr{} R^{r}f_*\Omega^{r}_{X/S}\xr{{\rm Tr}} \OO_S.  
\end{equation}
Let us now give the more precise formulation of Theorem \ref{thm-GM-simple}.

\begin{thm} \label{thm-GM-better}
Assumptions as in Theorem \ref{thm-GM-simple}. For all $i$, the morphism 
\begin{equation}\label{equation-eta-bar-tau}
  R^{r-i}f_*\Omega^{r}_{X/S} \xr{\text{\eqref{equation-duality-tau}}} (R^if_*\OO_X)^{\vee}\xr{\text{\eqref{equation-eta-bar}}} H^1_{{\rm inv}}(\Phi^i(X/S))
\end{equation} 
extends to a surjective morphism of $D'_S$-modules
\begin{equation}\label{equation-precise-morph}
 D'_S \cdot \epsilon(R^{r-i}f_*\Omega^{r}_{X/S}) \xr{} H^1_{{\rm inv}}(\Phi^i(X/S)).
\end{equation} 
\begin{proof}
Note that $R^if_*\OO_X=T\Phi^i(X/S)$, so that  \eqref{equation-eta-bar} makes sense.
We may suppose that $S=\Spec(R)$ is affine. Furthermore, we may suppose
$\Phi^i(X/S)={\rm Spf}(R[|x_1,\dots,x_d|])$ as formal schemes; in particular we already 
have a trivialization 
\begin{equation}\label{equation-triv-TPhi}
T\Phi^i(X/S)=\bigoplus_{j=1}^d \OO_Se_j.  
\end{equation}
Denote by $e_1^{\vee},\dots,e_d^{\vee}$ the dual basis, set 
\begin{align*}
 &\alpha^{\vee}_j:=\tau^{-1}(e_j^{\vee})\in \Gamma(R^{r-i}f_*\Omega^{r}_{X/S}), &\\
 &\omega_j:=\eta(e_j^{\vee})\in \Gamma(\Omega^1_{\Phi^i(X/S)/S}),
 &\bar{\omega}_j:=\bar{\eta}(e_j^{\vee})\in \Gamma(H^1(\Omega^*_{\Phi^i(X/S)/S})).
\end{align*}
If $P_1,\dots,P_d\in \Gamma(D'_S)$ satisfy 
$$
\sum_{j=1}^d P_j\cdot \epsilon(\alpha^{\vee}_j)=0 \quad \text{in $\Gamma(H^i_{{\rm dR}}(X/S))$,}
$$ 
then we need to show that 
\begin{equation}\label{equation-thm-GM-need-to-show}
\sum_{j=1}^d P_j\cdot \bar{\omega}_j=0 \quad \text{in $\Gamma(H^1(\Omega^*_{\Phi^i(X/S)/S}))$.}
\end{equation}
Our trivialization (\ref{equation-triv-TPhi}) induces an isomorphism
$$
\log:\Phi^i(X/S)\times_{S}S_{\Q} \xr{\cong} \hat{\mathbb{G}}_{a,S_{\Q}}^d,
$$
where $S_{\Q}$ is the base change to $\Spec(\Q)$ (see \ref{subsection-logarithm}). Moreover, we have $d\log^*(x_j)=\omega_j$ (Section \ref{subsubsection-log-and-invarian-one-forms}). Write 
$
\log^*(x_j)=\sum_{I}g_{j,I}x^I
$
with $g_{j,I}\in R\otimes_{\Z}\Q$.
Then (\ref{equation-thm-GM-need-to-show}) is equivalent 
to 
$$
f:=\sum_I \sum_{j=1}^dP_j(g_{j,I})x^{I}\in R[|x_1,\dots,x_d|]. 
$$
Recall from (\ref{equation-log-Cartier-modules}) that $\log$ induces
$$
\log:\Phi^i(X/S)(R[|x|])\xr{} \hat{\mathbb{G}}_{a}((R\otimes \Q)[|x|])^d=\bigoplus_{j=1}^d x\cdot (R\otimes \Q)[|x|].
$$
Moreover, the following diagram is commutative: 
$$
\xymatrix
{
\Phi^i(X/S)(R[|x|]) \ar[r]^{\log} \ar[d]^{\pi}
&
\hat{\mathbb{G}}_{a}((R\otimes \Q)[|x|])^d \ar[d]^{(xf_1,\dots,xf_d)\mapsto (f_1(0),\dots,f_d(0))}
\\
\Gamma(T\Phi^i(X/S))
&
\bigoplus_{j=1}^{d} Re_i, \ar[l]_{\text{~(\ref{equation-triv-TPhi})}}
}
$$
where $\pi$ is the projection. In view of Proposition \ref{proposition-Cartier-module-Wittvector-cohomology} we have 
$$\Phi^i(X/S)(R[|x|])=H^i(X,\mathbb{W}\OO_X)$$ 
and $\pi$ corresponds to the projection $H^i(X,\mathbb{W}\OO_X)\xr{} H^i(X,\OO_X)$ induced by the map $\mathbb{W}\OO_X\xr{} \OO_X$. 

Let $\xi\in \Phi^i(X/S)(R[|x|])$ and write $\log(\xi)=:(\log(\xi)_1,\dots,\log(\xi)_d)$. For $n\geq 1$ we obtain 
$$
\pi(F_n(\xi))=\sum_{j=1}^{d}(F_n(\log(\xi)_j))(0)\cdot e_j,
$$ 
because $\log$ is a morphism of $\mathbb{E}_R$-modules. Note that
$\pi(F_n(\xi))\in \Gamma(T\Phi^i(X/S))$ implies $(F_n(\log(\xi)_j))(0)\in R$.
By construction,
$
\log(\xi)_j=\xi^*\log^*(x_j)\in (R\otimes \Q)[|x|],
$
and we know that
$$
(F_n(\log(\xi)_j))(0)=n\cdot (\text{$n$-th coefficient of $\xi^*\log^*(x_j)$}).
$$
A theorem of Stienstra \ref{thm-stienstra} implies 
$$
\sum_{j=1}^{d} P_j\cdot (F_n(\log(\xi)_j))(0)\in nR,
$$
hence 
\begin{equation}\label{equation-GM-almost-final-step}
\sum_{j=1}^{d} P_j\cdot (\text{$n$-th coefficient of $\xi^*\log^*(x_j)$})\in R \qquad \text{for all $n$.}
\end{equation}
Now suppose that $\xi^*(x_j)\in \Z[|x|]$ for every $j=1,\dots,d$ (where the 
$x_j$ are now the coordinates from $\Phi^i(X/S)={\rm Spf}(R[|x_1,\dots,x_d|])$).
Then (\ref{equation-GM-almost-final-step}) implies that $\xi^*(f)\in R[|x|]$.
In view of Lemma \ref{lemma-check-coordinates} we finally obtain $f\in R[|x_1,\dots,x_d|]$.
\end{proof}
\end{thm}

\begin{lemma}\label{lemma-check-coordinates}
Let $R$ be flat over $\Z$. Let $f\in (R\otimes_{\Z}\Q)[|x_1,\dots,x_d|]$. Suppose
that $\xi^*(f)\in R[|x|]$ for all $\xi:\Spf(R[|x|])\xr{}\Spf(R[|x_1,\dots,x_d|])$ such that $\xi^*(x_j)\in \Z[|x|]$ for all $j$. Then $f\in R[|x_1,\dots,x_d|]$. 
\begin{proof}
  We prove by induction on $d$. The case $d=1$ is trivial. By subtracting 
an element in $R[|x_1,\dots,x_d|]$ we may assume that
$$
f=x_d^m\cdot \sum_{i=0}^{\infty} f_i(x_0,\dots,x_{d-1})x_d^{i}
$$
and $f_0\not\in R[|x_0,\dots,x_{d-1}|]$. For every $\xi':\Spf(R[|x|])\xr{}\Spf(R[|x_1,\dots,x_{d-1}|])$ such that $\xi'^*(x_j)\in \Z[|x|]$ for all $j$, and every 
integer $n\geq 1$, we define
$\xi$ by $\xi^*(x_j)=\xi'^*(x_j)$ if $j<d$ and $\xi^*(x_d)=x_d^n$. Let $n$
tend to $\infty$ in order to obtain $\xi'^*(f_0)\in R[|x_1,\dots,x_{d-1}|]$.
\end{proof}
\end{lemma}  

As in Corollary \ref{corollary-GM-even-simpler} we obtain the following statement.

\begin{corollary}\label{corollary-GM-even-simpler-isom?}
Let $f:X\xr{} S$ be a smooth, projective morphism of relative dimension $r$.
Suppose that $S$ is smooth over $\Spec(\Z)$ and suppose that 
$R^jf_*\OO_X$ is locally free for all $j$. Then, for all $i$, the sheaf 
$H^1_{{\rm inv}}(\Phi^i(X/S))\otimes_{\OO_S}\OO_{S_{\Q}}$ is a quotient of the 
submodule
$D_{S_{\Q}}\cdot R^{r-i}f_*\Omega^r_{X_{\Q}/S_{\Q}}$ of $H^{2r-i}_{{\rm dR}}(X_{\Q}/S_{\Q})$.
\end{corollary}

\begin{remark}
We don't know examples where the quotient map 
$$D_{S_{\Q}}\cdot R^{r-i}f_*\Omega^r_{X_{\Q}/S_{\Q}}\xr{} H^1_{{\rm inv}}(\Phi^i(X/S))\otimes_{\OO_S}\OO_{S_{\Q}}$$ 
is not an isomorphism.
\end{remark}

\subsection{}
In the following we recall the theorem of Stienstra that is used in the proof 
of Theorem \ref{thm-GM-better}. Stienstra proved it in \cite[Theorem~4.6]{Stienstra} 
by using his definition of the big de Rham-Witt complex. For the convenience
of the reader we will recall the proof and work with the big de Rham-Witt 
complex introduced in \cite{H}. 

\begin{thm}[\text{\cite[Theorem~4.6]{Stienstra}}] \label{thm-stienstra}
  Let $f:X\xr{} S$ be a smooth, projective morphism of relative dimension $r$.
Suppose that $S=Spec(R)$ is smooth over $\Spec(\Z)$ and suppose that 
$R^jf_*\Omega^i_{X/S}$ is free for all $i,j$. Fix an integer $m\geq 0$.
Take a basis $\omega_1,\dots,\omega_h$ of $H^m(X,\OO_X)$. Let $\omega_1^{\vee},\dots,\omega_h^{\vee}$ be the dual basis of $H^{r-m}(X,\Omega^r_{X/S})$. Take 
$\xi\in H^m(X,\mathbb{W}\OO_X)$ and define for every positive integer $n$ the 
elements $B_{n,1},\dots,B_{n,h}\in R$ by 
$$
\pi (F_n\xi)= \sum_{j=1}^h B_{n,j}\omega_j,
$$
where $\pi:H^m(X,\mathbb{W}\OO_X)\xr{} H^m(X,\OO_X)$ is the projection.

Suppose $P_1,\dots,P_h\in \Gamma(D'_S)$ are such that
$$
\sum_{j=1}^{h} P_j\omega_j^{\vee}=0 \qquad \text{in $H^{2r-m}(X,\Omega^*_{X/S})$,}
$$ 
then 
$$
\sum_{j=1}^h P_jB_{n,j}=0 \mod n.
$$
\begin{proof}
  For a truncation set $T$, let $\mathbb{W}_T\Omega_X^*$ be the big de-Rham Witt complex
defined in \cite{H}. We will only need a relative version: set 
$$
\mathbb{W}\Omega^*_{X/S}:=\varprojlim_{T} \mathbb{W}_T\Omega_X^*/(f^{-1}\mathbb{W}_T\Omega_S^1\cdot \mathbb{W}_T\Omega_X^*),
$$
where $T$ runs over all finite truncation sets. We denote by $\tau$ the 
projection $\mathbb{W}\Omega^*_{X/S}\xr{} \Omega^*_{X/S}$, it is a
morphism of differential graded algebras.  For all $n\geq 1$ we have 
the Frobenius $F_n:\mathbb{W}\Omega^*_{X/S}\xr{} \mathbb{W}\Omega^*_{X/S}$ 
satisfying  $dF_n=n\cdot F_nd$. Moreover, there is a natural morphism 
$$
\eta_{X/S}:\mathbb{W}\OO_X \xr{} \mathbb{W}\Omega^*_{X/S},
$$
that identifies $\mathbb{W}\Omega^0_{X/S}$ with $\mathbb{W}\OO_X$, but $\eta_{X/S}$
is in general not a morphism of complexes. 

For all integers $n\geq 1$, we set 
$$
\Z/n\otimes \mathbb{W}\Omega^*_{X/S} := \coker(\mathbb{W}\Omega^*_{X/S}\xr{n\cdot} \mathbb{W}\Omega^*_{X/S}), 
$$
and similarly we define $\Z/n\otimes \Omega^*_{X/S}$; we denote by $\epsilon_n$ the quotient map $\mathbb{W}\Omega^*_{X/S}\xr{} \Z/n\otimes \mathbb{W}\Omega^*_{X/S}$. The rule $dF_n=nF_nd$ implies that we obtain a commutative diagram of morphisms
of complexes
$$
\xymatrix{
\mathbb{W}\OO_X\ar[rr]^{\epsilon_n\circ \eta_{X/S}\circ F_n} \ar[rd]_{\epsilon_n\circ \eta_{X/\Z}\circ F_n}
&
&
\Z/n\otimes \mathbb{W}\Omega^*_{X/S}\\
&
\Z/n\otimes \mathbb{W}\Omega^*_{X/\Z}\ar[ur]
&
}
$$
We set
$$
\zeta_n:=(\tau\circ \epsilon_n\circ \eta_{X/S}\circ F_n)(\xi)\in H^m(X,\Z/n\otimes \Omega^*_{X/S});
$$
note that $H^m(X,\Z/n\otimes \Omega^*_{X/S})=H^m_{{\rm dR}}(X/S)\otimes_{R}R/n$, because 
the Hodge to de Rham spectral sequence degenerates at $E_1$ and hence $H^*_{{\rm dR}}(X/S)$
is a free $R$-module.
Recall that $H^{r-m}(X,\Omega^r_{X/S})\subset H^{2r-m}_{{\rm dR}}(X/S)$, as a next step
we need to prove 
\begin{equation}\label{equation-thm-stienstra-Bn}
  \langle \zeta_n, \omega_j^{\vee}\rangle = B_{n,j} \mod nR,
\end{equation}
via the pairing 
$$
\langle.,.\rangle: H^m_{{\rm dR}}(X/S)\times H^{2r-m}_{{\rm dR}}(X/S) \xr{} H^{2r}_{{\rm dR}}(X/S)=H^r(X,\Omega^r_{X/S})\xr{{\rm Tr}} R.
$$
Indeed, we have 
$$
\langle\zeta,\omega^{\vee}_j \rangle = (q(\zeta),\omega_j^{\vee}) \quad \text{for all $\zeta\in H^m_{{\rm dR}}(X/S)$,}
$$
with $q:H^m_{{\rm dR}}(X/S)\xr{} H^m(X,\OO_X)$ induced by the projection $\Omega^*_{X/S}\xr{} \OO_X$, and $(.,.)$ is the Grothendieck duality pairing (\ref{equation-Grothendieck-duality}). Now the equality $q(\zeta_n)=\pi(F_n\xi)$ modulo $nR$ implies (\ref{equation-thm-stienstra-Bn}). 

Since $\zeta_n$ is the image of the  class 
$(\tau\circ \epsilon_n\circ \eta_{X/\Z}\circ F_n)(\xi) \in  
H^m(X,\Z/n\otimes \Omega^*_{X/\Z})$, it is horizontal for the Gauss-Manin connection. 
Therefore $D(\zeta_n)=0$ for all derivations $D$ of $R$. From the compatibility 
of $\langle.,.\rangle$ with the Gauss-Manin connection  we get 
$$
D(B_{n,j})=\langle D(\zeta_n),\omega^{\vee}_j\rangle + \langle \zeta_n,D(\omega_j^{\vee})\rangle= \langle \zeta_n,D(\omega_j^{\vee})\rangle \mod nR
$$
for all derivations. This implies the theorem.
\end{proof}
\end{thm}

\begin{example}\label{example-GM}
  The logarithm of the formal group $\Phi^{n-r}(X/R)$ attached to a complete
intersection $X\subset \P^n_R$ of codimension $r$ has been computed by 
Stienstra \cite{Stienstra2}.

The classical example for Theorem \ref{thm-GM-simple} is the Legendre family
$X=\{zy^2-x(x-z)(x-\lambda z)=0\}\subset \P^2_R$ of elliptic curves, with $R=
\Z[\lambda][\frac{1}{2\lambda(1-\lambda)}]$. An invariant one-form of $\Phi^1(X/R)$
is given by
$$
\omega=\sum_{\substack{n=0\\ \text{$n$ even}}}^{\infty} \binom{n}{\frac{n}{2}}\sum_{k=0}^{\frac{n}{2}} \binom{\frac{n}{2}}{k}^2 \lambda^k x^ndx
$$
with the coordinate $x$ from \cite{Stienstra2} ($\Phi^1(X/R)$ is one dimensional). As $D'_{R}$-module, $H^1_{{\rm dR}}(X/R)$
is annihilated by 
$$
D:=\lambda(1-\lambda)(\frac{d}{d\lambda})^2+(1-2\lambda)\frac{d}{d\lambda}-\frac{1}{4}
$$ 
\cite{Katz}. Since $\binom{n}{\frac{n}{2}}\equiv \pm 1 \mod n+1$, Theorem \ref{thm-GM-simple} implies that 
$$
D(\sum_{k=0}^{\frac{n}{2}} \binom{\frac{n}{2}}{k}^2 \lambda^k)=0 \mod n+1,
$$
which the reader may prove easily by a direct computation (see \cite{Katz}).
The sequence
$(\sum_{k=0}^{\frac{n}{2}} \binom{\frac{n}{2}}{k}^2 \lambda^k)_{n+1}$
forms an element in 
$$
F\in \varprojlim_{\substack{n+1\\ \text{$n$ even}}} \Z/(n+1)[\lambda]\subset (\prod_{p\neq 2}\Z_p)[|\lambda|]
$$
(the projective system is ordered by division) 
that 
satisfies $D(F)=0$ and $F(0)=1$; thus it is given by the hypergeometric series
$F={}_2F_1(\frac{1}{2},\frac{1}{2};1,\lambda).$ 
Since $F$ suffices to reconstruct the operator $D$, the quotient morphism from 
Corollary \ref{corollary-GM-even-simpler-isom?} is an isomorphism in this case.

The analogous statements can also be proved for the Dwork family $\{\sum_{i=0}^mx_i^{m+1}-(m+1)\lambda x_0\cdots x_m=0\}\subset \P^m_R$ of Calabi-Yau varieties (see \cite{Yu}). 
\end{example}

\appendix

\section{Commutative formal Lie groups}

In this section we recall some basic facts in the theory of commutative 
formal Lie groups, that we need in the main text of the paper.  

\subsection{Invariant $1$-forms}\label{subsection-invariant-1-forms}

\subsubsection{}
Let $S$ be a scheme and $\mathfrak{X}$ be a commutative 
formal Lie group over $S$ (see \ref{subsubsection-formal-Lie-group}). 

\begin{definition}
  A $1$-form $\omega\in \Omega_{\mathfrak{X}/S}$ is called \emph{invariant}
if 
$$
m^*(\omega)={\rm pr}_1^*(\omega)+{\rm pr}_2^*(\omega),
$$
where $m:\mathfrak{X}\times_S \mathfrak{X}\xr{} \mathfrak{X}$ is the 
multiplication and ${\rm pr}_i$ are the projections.
\end{definition}

For example, $dx$ is an invariant $1$-form for $\hat{\mathbb{G}}_a$, and 
$dx/(1-x)$ is an invariant $1$-form for $\hat{\mathbb{G}}_m$ in the coordinates
from \ref{subsubsection-formal-completion}. Obviously, the invariant 
$1$-forms form a $\OO_S$-submodule of $\Omega_{\mathfrak{X}/S}$.

In \ref{subsubsection-formal-Lie-group} we introduced the tangent space
of $\mathfrak{X}$ by 
$$
T\mathfrak{X}(U):=\ker(\mathfrak{X}(\Spec(\OO_U[\epsilon]/\epsilon^2))\xr{} \mathfrak{X}(U)),
$$
where $U\xr{} \Spec(\OO_U[\epsilon]/\epsilon^2)$ is the morphism over $U$ given 
by $\epsilon\mapsto 0$. 
Given $\xi\in T\mathfrak{X}(U)$, we can define a derivation 
$D_{\xi}:\Gamma(U,\OO_{\mathfrak{X}})\xr{} \Gamma(U,\OO_S)$ by 
$$
D_{\xi}(f)\epsilon = \xi^*(f)-f(0).
$$
Here and in the following we consider $\OO_S$ as an $\OO_{\mathfrak{X}}$-module via the zero section $f\mapsto f(0)$. This induces an isomorphism of $\OO_S$-modules
\begin{equation}\label{equation-tangent-space-1-forms}
T\mathfrak{X}\xr{} \mathit{Der}_{\OO_S}(\OO_{\mathfrak{X}},\OO_S)=\mathit{Hom}_{\OO_{\mathfrak{X}}}(\Omega_{\mathfrak{X}/S},\OO_S).  
\end{equation}

\begin{proposition}\label{proposition-invariant-forms}
There is a unique morphism of $\OO_S$-modules 
$$
\eta: T\mathfrak{X}^{\vee} \xr{} \Omega_{\mathfrak{X}/S},
$$  
satisfying the following properties:
\begin{enumerate}
\item The morphism $\eta$ induces an isomorphism between $T\mathfrak{X}^{\vee}$
and the invariant $1$-forms. 
\item The diagram 
$$
\xymatrix
{
T\mathfrak{X}^{\vee} \times T\mathfrak{X}\ar[rr] \ar[d]^{\eta\times id}
&
&
\OO_S \ar[d]^{=}
\\
\Omega_{\mathfrak{X}/S} \times T\mathfrak{X} \ar[rr]^{\text{~(\ref{equation-tangent-space-1-forms})}}
&
&
\OO_S
}
$$
is commutative.
\end{enumerate}
\begin{proof}
  \cite[1.19]{Zink}
\end{proof}
\end{proposition}
If $\mathfrak{X}={\rm Spf}(\OO_S[|x_1,\dots,x_d|])$ as formal scheme over $S$ then 
$$
\Omega_{\mathfrak{X}/S}\xr{} \OO_S^d, \quad \sum_{i=1}^d f_i dx_i\mapsto (f_1(0),\dots,f_d(0)),
$$
induces an isomorphism between the invariant $1$-forms and $\OO_S^d$.

\begin{lemma}\label{lemma-invariant-implies-closed}
  Let $S$ be a scheme that is flat over $\Spec(\Z)$. Let $\mathfrak{X}$ be a 
commutative formal Lie group  over $S$. Then the invariant $1$-forms are closed.
\begin{proof}
  We may suppose $S=\Spec(R)$. Denoting by $S\otimes_{\Z}\Q$  the base change 
to $\Spec(\Q)$,
  we have 
$\Gamma(S,\Omega^2_{\mathfrak{X}/R})\subset 
\Gamma(S\otimes_{\Z}\Q,\Omega^2_{\mathfrak{X}\otimes_{\Z}\Q/R\otimes_{\Z}\Q})$. Since
$\mathfrak{X}\otimes_{\Z}\Q$ is isomorphic to the formal completion of 
a vector bundle (Section \ref{subsection-logarithm}) we are reduced to the 
obvious case $\hat{\mathbb{G}}_a$.
\end{proof}
\end{lemma}

\subsection{Logarithms and formal groups in characteristic zero}
\label{subsection-logarithm}

\subsubsection{}
Let $S$ be a scheme. Recall from  Section \ref{subsubsection-formal-Lie-group} 
that we have a functor 
\begin{equation}\label{equation-T}
T:\text{(Comm. formal Lie groups over $S$)}\xr{} 
\text{(loc. free sheaves on $S$ of finite rank)},  
\end{equation} 
that assigns to a group its tangent space (or Lie algebra). 

\begin{proposition}\label{proposition-T-Q}
Suppose that $S$ is a scheme over $\Spec(\Q)$. Then $T$ (\ref{equation-T}) is   
an equivalence of categories. 
\begin{proof}
  \cite[4.7]{Zink}
\end{proof}
\end{proposition}

For every locally free sheaf $\mathcal{E}$ of finite rank on $S$ we have the associated vector
bundle 
$$
V(\mathcal{E}):=\Spec({\rm Sym}^*\mathcal{E}^{\vee}).
$$
The formal completion $\hat{V}(\mathcal{E})$ yields a commutative formal 
Lie group with tangent space $\mathcal{E}$. If $S$ is a scheme over $\Q$, then
for every commutative formal Lie group $\mathfrak{X}$ there exists a unique
isomorphism of formal groups 
$$
\mathfrak{X}\xr{} \hat{V}(T\mathfrak{X}),
$$ 
such that the induced map on the tangent spaces is the identity; the functor 
$\hat{V}$ is inverse to $T$. For every morphism 
$\phi:T\mathfrak{X}\xr{} \mathcal{E}$ we will denote by 
$$
\log_{\phi}: \mathfrak{X}\xr{} \hat{V}(\mathcal{E}) 
$$
the unique morphism over $\phi$.

In general we identify $\hat{V}(\mathcal{\OO}_S^d)$ with $\hat{\mathbb{G}}_a^d$. 
A basic example is 
$\mathfrak{X}=\hat{\mathbb{G}}_m$, $\mathcal{E}=\OO_S$, and $\phi$ coming 
from the coordinate of $\hat{\mathbb{G}}_m$ in \ref{subsubsection-formal-completion}.
Then $\log_{\phi}:\hat{\mathbb{G}}_m\xr{} \hat{\mathbb{G}}_a$ is given by    
$\log_{\phi}^*(x)=\sum_{k=1}^{\infty} \frac{x^k}{k}$. 

\subsubsection{}\label{subsubsection-log-and-invarian-one-forms}
Again, suppose that $S$ is a scheme over $\Spec(\Q)$. The canonical 
morphism $\mathfrak{X}\xr{} \hat{V}(T\mathfrak{X})$ can be constructed as follows.
By  Proposition \ref{proposition-invariant-forms} we have a canonical 
morphism $\eta:T\mathfrak{X}^{\vee}\xr{} \Omega_{\mathfrak{X}/S}$ that induces an isomorphism
with the invariant forms. Since invariant forms are closed (Lemma \ref{lemma-invariant-implies-closed})
we obtain a morphism of $\OO_S$-modules
$$
d^{-1}\circ \eta: T\mathfrak{X}^{\vee} \xr{} \ker(0^*:\OO_{\mathfrak{X}}\xr{} \OO_S).
$$ 
This morphism induces $\mathfrak{X}\xr{} \hat{V}(T\mathfrak{X})$.

\subsubsection{}
Let $S=\Spec(R)$ be flat over $\Spec(\Z)$. 
Let $\mathfrak{X}$ be a commutative formal Lie group. Suppose that $\mathfrak{X}={\rm Spf}(R[|x_1,\dots,x_d|])$ as formal schemes 
and that the zero is defined by $x_i\mapsto 0$.  
We denote by $\mathfrak{X}_{\Q}$ the base change to 
$S_{\Q}:=S\times_{\Spec(\Z)}\Spec(\Q)$. 

Since $T\mathfrak{X}$ comes equipped with a trivialization, we have a 
canonical map
\begin{equation}\label{equation-log-Cartier-modules}
  \mathfrak{X}(R[|x|])\xr{} \mathfrak{X}_{\Q}((R\otimes \Q)[|x|])\xr{\log} 
\hat{\mathbb{G}}_a^d((R\otimes \Q)[|x|])=\bigoplus_{i=1}^d x\cdot (R\otimes \Q)[|x|].
\end{equation}

 

 

\end{document}